\documentclass[preprint]{elsarticle}

\usepackage{amssymb,graphicx}
\usepackage{latexsym}

\usepackage{hyperref}
\usepackage{soul}
\setulcolor{red}
\sethlcolor{yellow}
\usepackage{color}
\usepackage{srcltx}
\newbox\barleftbox
\newbox\barrightbox
{\setlength\fboxsep{0pt}\setlength\fboxrule{0pt}
 \global\setbox\barleftbox=\hbox{%
   \colorbox[gray]{.8}{\hskip 4pt\relax\vrule height3pt width 0pt}%
   \hskip4pt}
 \dimen0=\ht\barleftbox \advance\dimen0-1pt
 \ht\barleftbox=\dimen0
 \dp\barleftbox=-1pt
 \global\setbox\barrightbox=\hbox{%
   \hskip4pt
   \colorbox[gray]{.8}{\hskip 4pt\relax\vrule height3pt width 0pt}}
 \dimen0=\ht\barrightbox \advance\dimen0-1pt
 \ht\barrightbox=\dimen0
 \dp\barrightbox=-1pt
}

\newproof{proof}{Proof}



\newcommand{\leftstrip}[1]{%
   \valign{##\cr
           \leaders\copy\barleftbox\vfill\cr
           \vbox{\hsize\marginparwidth\advance\hsize-8pt
                 \raggedright\sffamily\footnotesize #1}\cr
   }
}
\newcommand{\rightstrip}[1]{%
   \valign{##\cr
           \vbox{\hsize\marginparwidth\advance\hsize-8pt
                 \raggedright\sffamily\footnotesize #1}\cr
           \leaders\copy\barrightbox\vfill\cr
   }
}

\let\oldmarginpar\marginpar
\renewcommand\marginpar[1]{%
  \oldmarginpar[\leftstrip{#1}]{\rightstrip{#1}}}

\usepackage{amsmath}

\usepackage{amsfonts}
\usepackage{graphics}

\usepackage{color}

\newtheorem{theorem}{Theorem}[section]
\newtheorem{lemma}[theorem]{Lemma}
\newtheorem{e-proposition}[theorem]{Proposition}
\newtheorem{corollary}[theorem]{Corollary}
\newtheorem{e-definition}[theorem]{Definition}
\newtheorem{remark}{Remark}[section]
\newtheorem{example}{Example}[section]
\newtheorem{theoreme}{Th\'eor\`eme}[section]

\newtheorem{definition}[theoreme]{Definition}

\numberwithin{equation}{section}

\long\def\salta#1{\relax}

 \def\1{\raisebox{2pt}{\rm{$\chi$}}}

\def\intr{\int_{B_{R}}}

\def\z{{z}}
\def\up{u_{p}}
\def\rife#1{(\ref{#1})}
\newcommand{\hide}[1]{}

\newcommand{\R}{{\mathbb R}}

\renewcommand{\H}{{\mathcal H}}

\renewcommand{\epsilon}{{\varepsilon}}
\newcommand{\loc}{\mathrm{loc}}

\newcommand{\nada}[1]   {}

\newcommand{\res}               {\!\!\mathop{\hbox{
                                \vrule height 7pt width .5pt depth 0pt
                                \vrule height .5pt width 6pt depth 0pt}}
                                \nolimits}

\def\div{{\rm div}}

\def\bk{\color{black}}

\def\rn{\mathbb{R}^{N}}

\def\re{\mathbb{R}}

\def\be{\begin{equation}}
\def\ee{\end{equation}}

\def\vare{\varepsilon}

\def\dive{{\rm div}}

\def\into{\int_{\Omega}}
\def\intq{\int_Q}

\def\w-1p'{W^{-1,p'}(\Omega)}

\def\w-1pd{W^{-1,p'}(D)}
\def\pw-1p'{L^{p'}(0,T;W^{-1,p'}(\Omega))}
\def\l{\textsl{L}}

\def\dys{\displaystyle}

\def\lp'n{(L^{p'}(\Omega))^{N}}

\def\car#1{\raise2pt\hbox{$\chi$}_{#1}}

\def\lp'n{(L^{p'}(\Omega))^{N}}

\def\t1p0{T^{1,p}_{0}(\Omega)}
\def\w-1p'{W^{-1,p'}(\Omega)}
\def\pw-1p'{L^{p'}(0,T;W^{-1,p'}(\Omega))}

\def\lil2{L^{\infty}(0,T;L^2 (\Omega))}

\def\l2h10{L^2 (0,T ; H^1_0 ( \Omega ))}

\def\m2{M^{\frac{N(p-1)}{N-1}}(\Omega)}

\def\l{\textsl{L}}

\def\vare{\varepsilon}

\def\into{\int_{\Omega}}

\def\intq1{\displaystyle \int_{\Omega \times (0, 1)}}

\def\dys{\displaystyle}

\def\m{\noalign{\medskip}}

\begin{document}

\title{Large solutions for the elliptic $1$-Laplacian with absorption}

\date{\today}
\author[UV]{Salvador  Moll}
\ead{j.salvador.moll@uv.es}

\author[sbai]{Francesco  Petitta\corref{cor1}}
\ead{francesco.petitta@sbai.uniroma1.it }

\address[UV]{
Departament d'An\`{a}lisi Matem\`atica,
Universitat de Val\`encia, Valencia, Spain.}
\address[sbai]{Dipartimento S.B.A.I,
Sapienza, Universit\`{a} di Roma, Italy}

\cortext[cor1]{Corresponding author}

\begin{abstract}
In this paper we give a general condition on the absorption term of the 1-Laplace elliptic equation for the existence  of suitable large solutions. This condition can be considered as the correspondent Keller-Osserman condition for the $p$-Laplacian, in the case $p=1$. We also provide  conditions that guarantee uniqueness for solutions to such problems.
\end{abstract}


\maketitle

\tableofcontents

\section{Introduction}

The study of both existence and uniqueness of large solutions for nonlinear partial differential equations with absorption goes back to the pioneering papers by Keller (\cite{ke}) and Osserman (\cite{os}). Starting from there a huge literature has been devoted to the study of such a problems in particular for their connection to several branches of   mathematics as Differential Geometry, Probability, and Control Theory (see for instance \cite{ln,dy,lg,ll}).

From the mathematical point of view the main idea is that, roughly, the existence of solutions (to a  certain  PDE) that blow up on the boundary of a domain  is strictly related to the absorption role played by suitable lower order terms. Far to provide a complete list of references we refer to \cite{dl,mat,mm,gp,dls,vero,gr,lp,lieber, MP} and references therein for a review on the subject.   To be a bit more precise, let $\Omega$ be a bounded {domain} of $\R^N$ with Lipschitz boundary. The existence of a large solution for problem
\begin{equation}\label{p-laplace-u-q}
\begin{cases}
   \Delta_p u  = u^q  & \text{in}\ \Omega,\\
 u=+\infty &\text{on}\ \partial\Omega,
  \end{cases}
\end{equation}
can be proved provided $q>p-1$ (see for instance \cite{dl}).

For a general increasing and continuous nonlinearity $f$, with $f(0)=0$, problem \begin{equation}
  \label{p-laplace-f} \begin{cases}
   \Delta_p u  = f(u)  & \text{in}\ \Omega,\\
 u=+\infty &\text{on}\ \partial\Omega,
  \end{cases}
\end{equation} has a solution if and only $f$ satisfies the so-called Keller-Osserman condition; \begin{equation}
  \label{KOp} \int_1^\infty \frac{1}{F(s)^{\frac{1}{p}}}\,ds<+\infty\,,
\end{equation} with $F(s)=\int_0^s f(t)\,dt$.

\medskip

In this paper we deal with both existence and uniqueness of solutions to the $1$-Laplace problem  with absorption \begin{equation}
  \label{maineq} \left\{\begin{array}
    {cc} \displaystyle \Delta_{1} u :=\div\left(\frac{Du}{|Du|}\right)=f(u)\quad & {\rm in \ } \Omega \\ \\ u=+\infty & {\rm on \ }\partial\Omega
  \end{array}\right.
\end{equation} under different conditions on both $f$ and $\Omega$. {The study of problems involving 1-Laplace type operators arises, for instance, in the study of image restoration as well as in some optimal design problems in the theory of torsion (see for instance \cite{ka,Sapiro,ACMBook} and references therein for a review on the main applications).
A systematic  mathematical study of this type of problems began with the works \cite{ACM-N,ACM} and a comprehensive treatment of the $1$-Laplace diffusive term can be found in the monograph \cite{ACMBook}.}

\medskip
We will call our solutions \emph{Large Solutions} in order to be consistent with the existing literature,  though, as we will see, the  boundary condition $u=+\infty$ should be understood in a very weak sense (see also Remark \ref{largemax} later). In fact, as we will see, depending on the assumptions (to be specified later) on $f$ and $\Omega$, very different situations occur: large solutions turn out to be, in many cases, globally bounded and they can be regarded as \emph{maximal solutions} for the equation in \rife{maineq}.    Anyway, depending, respectively,  on the boundary regularity of $\Omega$ and on the behavior of  $f$, then the value $+\infty$ can be attained at some points (Remark \ref{infinito}) of $\partial\Omega$ {or} maximality can be lost (Remark \ref{maxi}). {We want to stress that this range of phenomena  is intrinsic to the 1-Laplacian operator. In fact, as first observed in Osher and Sethian's celebrated paper \cite{OsSe} (see also the monograph \cite{OsFe}), the 1-Laplacian operator is closely related to the mean curvature operator in the following way: consider the surface given by the level set $\{u(x)=k\}$; then  its unit normal is formally given by ${\bf n}(x)=\frac{Du}{|Du|}$. Therefore, the mean curvature of the surface at the point $x$ is formally given by $${\bf H}(x)=\div ({\bf n})(x)=\div \left(\frac{Du}{|Du|}\right)(x);$$i.e. the 1-Laplacian operator. This relationship clearly shows that the behavior at the boundary $\partial \Omega$ of the solutions to \eqref{maineq} might depend on the mean curvature of the boundary itself, and in particular, as we will see, on its boundedness.}

\medskip
Though both Neumann and Dirichlet  1-Laplace  type problems with absorption terms have been considered (see for instance \cite{de1} and references therein)  we want to stress that, as a by-product  of our arguments, existence and uniqueness for suitable nonhomogeneous Dirichlet boundary value problems will also be obtained (see Remark \ref{existencedirichlet}), some of these, to our knowledge,  being missed in the existing literature.

\medskip

 Let us describe the structure of the paper. \bk After a preliminary section where we recall some basic definitions and results,  and we define the notation we shall use throughout the paper, in Section \ref{sec.unique} we define a notion of solution to \eqref{maineq} for the case of $\Omega$ having a Lipschitz boundary and under very general assumptions on the nonlinearity $f$. We prove that, if it exists, this type of solution is maximal among all distributional solutions and, as a consequence, we obtain uniqueness for such solutions.

In Section \ref{sec.p1} we give the existence result for solutions to problem \eqref{maineq}.
The formal idea is the following one: after the change of variables $v=f(u)$, then, using the homogeneity of the 1-Laplacian operator,  \eqref{maineq} formally transforms into \begin{equation}
  \label{maineqv} \left\{\begin{array}
    {cc} \displaystyle \Delta_{1}v =v\quad & {\rm in \ } \Omega \\ \\ v=+\infty & {\rm on \ }\partial\Omega
  \end{array}\right.
\end{equation}
We show that this procedure is correct in case of the domain satisfying a uniform interior ball condition or being a convex body.

Existence and uniqueness of solutions to \eqref{maineqv} can be immediately derived from the recent work \cite{MP}, where the main issue is the study of both existence and uniqueness of large solutions for parabolic problems without lower order absorption terms whose model is
 \begin{equation}\label{maine}
\begin{cases}
    u_t=\Delta_p u  & \text{in}\ Q_T,\\
    u=u_0  & \text{on}\ \{0\} \times\Omega,\\
 u=+\infty &\text{on}\ (0,T)\times\partial\Omega,
  \end{cases}
\end{equation}
where $u_0\in L^1_{loc}(\Omega)$ is a nonnegative function and $\Omega$ is a bounded open subset of $\rn$ with smooth boundary and $1\leq p<2$. In this case, the na\"ive  idea is that, clearly, the term $u_{t}$ plays itself an absorption role of growth {of order}$1$ and this allows to avoid, for instance, the interior blow up of the approximating solutions, at least, for small $p$.

 The results about \eqref{maineqv} allow us to  show that \eqref{maineq} has a unique solution in the case of $f$ being an increasing function defined on $\re$ with $f^{-1}$ locally Lipschitz continuous in $]0,+\infty[$. This assumption takes the place of the Keller-Osserman condition for higher growth operators. As a matter of fact,  the $1$-Laplacian is singular enough to play {an absorption role by} itself, thanks to its homogeneity property. So that, though, as we said, solutions can attain the values $+\infty$ on the boundary, essentially no assumptions are needed on the behavior of $f(s)$ at infinity.

Finally, in Section \ref{Secpto1} we prove a stability result. Solutions to  \rife{maineq}, in the  case of  $\Omega$ being of class $C^2$ and  under natural Keller-Osserman type conditions on $f$ (that will be specified later), can be obtained through a  stability procedure by taking the limit as $p\to 1^{+}$ in \eqref{p-laplace-f}. %Let us remark that this stability  procedure, which has an independent interest, turns also to cover some particular cases not covered in the previous sections\footnote{es cierto? Decir cuales en seccion \ref{Secpto1}}.

\medskip
{ Summarizing, the different hypotheses on $f$ we will use are the following:
\begin{itemize}
  \item[$(H_1)$.] $f$ is increasing.
  \item[$(H_2)$.] $f$ is increasing, with $f^{-1}$ increasing and locally Lipschitz continuous in $]0,+\infty[$.
  \item[$(H_3)$.] $f$ is  {continuous and} increasing, $f(0)=0$ and it verifies $f(s)\geq c s^{\overline{q}}$ for some $c,\overline{q}>0$.
\end{itemize}
We sum up here the results obtained. The least assumption on the bounded set $\Omega$ is to have a Lipschitz boundary.
\begin{itemize}
  \item $(H_1)$ gives uniqueness of large solutions (Corollary  \ref{unique}). It is moreover a sharp condition (Remark \ref{maxi}).
  \item $(H_2)$ is a sufficient condition for existence of solutions in case $\Omega$ satisfies a uniform interior ball condition or it is a convex body (Theorem \ref{exgenf}). Therefore, $(H_2)$ can be considered as the corresponding Keller-Osserman condition for $p=1$.
  \item If $\Omega$ is a $C^2$ domain, then $(H_3)$  is enough to prove that solutions to \eqref{p-laplace-f} (which exist for small $p$'s since $(H_3)$ implies \eqref{KOp}) converge to the solution to \eqref{maineq} as $p\to 1^+$.
\end{itemize}
}
 \section{Preliminaires and notations}\label{preliminaris}
 In this section we collect the main notation and some useful results we will use in our analysis.

\subsection{Functions of bounded variations and some generalizations}

Let $\Omega$ be  an open subset of $\R^N$. A
function $u \in L^1(\Omega)$ whose gradient $Du$ in the sense of
distributions is a vector valued Radon measure with finite total
mass in $\Omega$ is called a {\it function of bounded variation}.
The class of such functions will be denoted by $BV(\Omega)$ and $|Du|$ will denote the total variation of the measure $Du$.

A measurable set $E\subset\R^N$ is said to be of finite perimeter in $\Omega$ if $\chi_E\in BV(\Omega)$. In this case, the perimeter of $E$ in $\Omega$ is defined as $Per(E,\Omega) := {|D\chi_E|(\Omega)}$. We shall use the notation $Per(E) := Per(E,\R^N)$. For sets of finite
perimeter E one can define the essential boundary $\partial^*E$, which is a countably $(N - 1)$-rectifiable set with
finite $\mathcal H^{N-1}$ measure, where $\mathcal H^{N-1}$ is the $(N - 1)$-dimensional Hausdorff measure. Moreover, the outer unit normal $\nu^E(x)$ exists at $\mathcal H^{N-1}$ almost all points $x$ of $\partial^*E$. It holds that
the measure $|D\chi_E|$ coincides with the
restriction of $\mathcal H^{N-1}$ to $\partial^*E$.

For further information and properties concerning functions of bounded variation and sets of finite perimeter
we refer to \cite{Ambrosio}, \cite{EG} or \cite{Ziemer}.

\medskip

 We consider the following truncature functions.
For $k>0$, let $T_{k}(s) := \max(\min(-k,s),k)$.
Given any function $u$ and $a,b\in\R$ we shall use the notation
$[u\geq a] = \{x\in \R^N: u(x)\geq a\}$,  and similarly for the sets %\footnote{unificar con $\{...\}$ ?}
 $[u
> a]$, $[u \leq a]$, $[u < a]$, etc.,  while the symbol $u\res{E}$ will indicate the restriction of the function $u$ to the measurable set $E\subset\rn$\bk.

\medskip
Given a real function $f(s)$, we define  its  positive part as  $f^+ (s)=\max(0,f(s))$.
{For our purposes, }we  need to consider the function spaces $$TBV(\Omega):= \left\{
u \in L^1(\Omega)  \ :  \ \ T_k(u) \in BV(\Omega), \ \ \forall \ k>0 \right\},$$ {$$TBV_{loc}(\Omega):= \left\{
u \in L^1_{loc}(\Omega)  \ :  \ \ T_k(u) \in BV(\Omega), \ \ \forall \ k>0 \right\},$$ }
\noindent and to give a sense to the
Radon-Nikodym derivative (with respect to the Lebesgue measure)
$\nabla u$ of $Du$ for a function $u \in TBV_{loc}(\Omega)$.

\begin{lemma}\label{WRN}{\cite[Lemma 1]{ACM4:01}}
For every $u \in TBV_{loc}(\Omega)$ there exists a unique measurable
function $v : \Omega \rightarrow \R^N$ such that
\begin{equation}\label{E1WRN}
\nabla T_k(u) = v \1_{[| u | < k]} \ \ \ \ \ {\mathcal
L}^N-{\rm a.e.}, \ \ \forall \ k>0.
\end{equation}
\end{lemma}

Thanks to this result we define $\nabla u$ for a function $u \in TBV_{loc}(\Omega)$ as the
unique function $v$ which satisfies (\ref{E1WRN}). {Obviously, if $w\in W_{\rm loc}^{1,1}(\Omega)$, then the generalized gradient turns out to coincide with the classical distributional one}. This notation will be used throughout
in the sequel.

\medskip

 We denote by ${\mathcal P}$ the set of nondecreasing Lipschitz continuous functions ${S} : [0, +\infty[ \rightarrow \R$
satisfying $S^{\prime}(s) = 0$ for $|s|$ large enough.
We recall the following result.

\begin{lemma}\label{CR}{\cite[Lemma 2]{ACM4:01}}
If $u \in TBV(\Omega)$, then $S(u) \in BV(\Omega)$ for every $S \in {\mathcal P}$. Moreover, $\nabla
S(u) = S^{\prime}(u) \nabla u$ \ ${\mathcal L}^N$-{\rm a.e.}
\end{lemma}

\subsection{A generalized Green's formula}\label{GreenAnz}

We shall need several results from \cite{Anzellotti1} (see also \cite{ACMBook}). Let
$$ X(\Omega) = \left\{ \z \in L^{\infty}(\Omega; \R^N) \
:
 \ \div(\z)\in L^1(\Omega) \right\}
$$

 If $\z \in X(\Omega)$ and $w \in BV(\Omega) \cap
L^{\infty}(\Omega)$
we define the functional $(\z,Dw):
C^{\infty}_{0}(\Omega) \rightarrow \R$ by the formula
\begin{equation}\label{defmeasx1}
\langle (\z,Dw),\varphi\rangle := - \int_{\Omega} w \, \varphi \,
\div(\z) \, dx - \int_{\Omega} w \, \z \cdot \nabla \varphi \, dx.
\end{equation}
In \cite{Anzellotti1} it is proved that $(\z,Dw)$ is a Radon measure in $\Omega$  verifying
$$
 (\z,Dw)(\Omega) = \int_{\Omega} \z \cdot \nabla w \, dx \ \
\ \ \ \forall \ w \in W^{1,1}(\Omega) \cap L^{\infty}(\Omega).
$$

Moreover, for all $w\in BV(\Omega)\cap L^\infty(\Omega)$,  $(\z,Dw)$ is absolutely continuous with respect to the total variation of $ D w$ and it holds,
\begin{equation}\label{acota}
(\z,Dw)(B)=\int_\Omega \theta(\z,Dw)|Dw| \leq \|\z\|_{\infty} |Dw|(B),
\end{equation}
for any Borel set $B\subseteq\Omega$, where $\theta(\z,Dw)$ is the Radon-Nikodym derivative of $(\z,Dw)$ with respect to $|Dw|$.

We will use the following result which is derived again from the work in \cite{Anzellotti1}.

\begin{lemma}
  \label{anz} Let $w\in BV(\Omega)\cap L^\infty(\Omega)$ and $\z\in X(\Omega)$. Then,\begin{itemize}
    \item[(a)] Coarea formula \cite[Proposition 2.7.(ii)]{Anzellotti1}: $$(\z,Dw)(B)=\int_{-\infty}^\infty (\z,D\chi_{[w>\lambda]})(B)\,d \lambda\,,\quad {\rm for \ any \ Borel \ set \ } B\subset\Omega.$$
    \item[(b)] Let $\alpha$ be an increasing function. Then, a slight modification of \cite[Proposition 2.8]{Anzellotti1} shows that if $\alpha(w)\in BV(\Omega)\cap L^\infty(\Omega)$, then $$\theta(\z,D(\alpha(w)))(x)=\theta(\z,D w)(x)\,,\quad |Dw|-{\rm \ a.e. \ in \ } \Omega.$$
  \end{itemize}
\end{lemma}

In \cite{Anzellotti1}, a weak trace on $\partial \Omega$ of the
normal component of  $\z \in X(\Omega)$, {denoted by $[\z, \nu]$, is defined.
Moreover, the
following {\it Green's formula}, relating the function $[\z, \nu]$
and the measure $(\z, Dw)$, for  $\z \in X(\Omega)$ and  $w \in
BV(\Omega) \cap L^{\infty}(\Omega)$ is established
\begin{equation}\label{Green}
\int_{\Omega} w \ \div (\z) \ dx + (\z, Dw)(\Omega) =
\int_{\partial \Omega} [\z, \nu] w \ d\H^{N-1}.
\end{equation}}

\section{Definition and uniqueness of large solutions}\label{sec.unique}
Throughout this section,  $f$ is considered to be an increasing function and $\Omega$ is a bounded set in $\R^N$ with Lipschitz boundary. We give the following definition of distributional solution to the equation in  \rife{maineq} which is the natural extension to the classical one (see \cite{ACM-N, ACMBook, MP} and references therein).

\begin{definition}\label{def.dist}
  We say  that $u\in TBV(\Omega)$ is a distributional solution of
   $$
   \div\left(\frac{Du}{|Du|}\right)=f(u)
   $$
    if there exists $\z\in X(\Omega)$, with $\|\z\|_{\infty}\leq 1$, such that $f(u)={\rm div } \z$ in $\mathcal{D}'(\Omega)$ and $(\z,DT_k (u))=|DT_k (u)|$ as measures, for any $k>0$.
\end{definition}

\begin{remark}\label{reminside}
  Observe that if $u\in {T}BV(\Omega)$ is a distributional solution of \eqref{maineq} in the sense of Definition \ref{def.dist}, then $u\res_{\Omega'}$ is a distributional solution of \eqref{maineq} for all $\Omega'\subset\Omega$ with Lipschitz boundary. In fact it suffices to take $\z':=\z\res_{\Omega'}$ since, by the fact that $(\z,DT_k(u))=|DT_k(u)|$ as measures and \eqref{acota}, then $(\z,DT_k(u))(B)=|DT_k(u)|(B)$ for any Borel set $B\subset\Omega$; in particular for $B=\Omega'$. Then, again by \eqref{acota}, $(\z',DT_k(u)\res_{\Omega'})=|DT_k(u)\res_{\Omega'}|$.
\end{remark}

Here is our definition of \emph{large solution} for problem \rife{maineq}.

\begin{definition}\label{defisol}
   We say that $u\in TBV(\Omega)$ is a (large) solution {to} \eqref{maineq} if there exists $\z\in X(\Omega)$, with $\|\z\|_{\infty}\leq 1$, such that \begin{equation}\label{distr}f(u)={\rm div } \z\,,\quad {\rm in \ }\mathcal{D}'(\Omega)\,,\end{equation} \begin{equation}\label{eqmeasures}(\z,DT_k(u))=|DT_k(u)|\,,\quad {\rm  as \ measures \ for \ any \ } k>0 {\rm \
      and \ }\end{equation} \begin{equation}\label{boundarycond}[\z,\nu]=1\,,\quad  {\rm for \ a.e. \ }x\in \partial\Omega\end{equation}
\end{definition}

\begin{remark}\label{largemax}
Let us make some comments about Definition \ref{defisol}. First of all, we note that
   in case that $v\in BV(\Omega) \cap L^{\infty}(\Omega)\bk$, then, by \eqref{acota}, condition \eqref{eqmeasures} is equivalent to  \begin{equation}\label{eqpairing2}(\z,D v)= |Dv|\quad {\rm as \ measures.}\end{equation}Secondly, observe that a large solution is nothing but a distributional solution verifying the boundary condition \eqref{boundarycond}. Observe that, since $\|\z\|_\infty\leq 1$, then \eqref{boundarycond} forces the vector field $\z$ to be parallel to the outward unit exterior normal to the boundary and to have its biggest possible magnitude at the boundary. This is the mild way in which condition \lq\lq$u=+\infty$''   must be understood. Also observe that solutions to Dirichet problems involving the 1-Laplacian as the diffusion term do not verify, in general, the boundary condition in a classical trace sense (see e.g. \cite{ACM}, \cite{MST}). Usually,  if the Dirichlet constraint is \lq\lq$u=g$ at $\partial\Omega$", with $g\in L^1(\partial\Omega)$, then this condition transforms into $$[\z,\nu]\in {\rm sign \ } (T_k(g)-T_k(u))\,,\quad \mathcal H^{N-1} {\rm \ at \ } \partial\Omega\,$$for any $k>0$, where {\rm sign} is the multivalued sign function.

We finally note that, for the parabolic case without absorption studied in \cite{MP}, the condition at the boundary for a large solution is the same as \eqref{boundarycond}. Moreover, this condition produces solutions to be maximal as the following result shows.
\end{remark}

 \begin{theorem}
    \label{maximal} Let $f:\R\to \R$ be {an increasing function}. If $u$ is a large solution of \eqref{maineq} and $\overline u$ is a distributional solution of \eqref{maineq}, then $u\geq \overline u$ for a.e. $x\in \Omega$
  \end{theorem}
  \begin{proof}
  By definition, there exist $z,\overline z\in X(\Omega)$ such that \begin{equation}
      \label{maximaleq1} f(u)={\rm div } z\,,
    \end{equation}
    \begin{equation}
      \label{maximaleq2} f(\overline u)={\rm div }\overline z.
    \end{equation} We multiply \eqref{maximaleq1} by $-(T_k(\overline u)-T_k(u))^+$, and integrate by parts in $\Omega$. We obtain $$-\int_{\Omega}(T_k(\overline u)-T_k(u))^+f(u)\stackrel{\eqref{boundarycond}}=\int_{\Omega}(z,D(T_k(\overline u)-T_k(u))^+)$$$$-\int_{\partial \Omega }(T_k(\overline u)-T_k(u))^+\, d\mathcal H^{N-1}\,.$$Similarly, $$\int_{\Omega}(T_k(\overline u)-T_k(u))^+f(\overline u)=-\int_{\Omega}(\overline z,D(T_k(\overline u)-T_k(u))^+)$$$$+\int_{\partial \Omega }[\overline z,\nu](T_k(\overline u)-T_k( u))^+\, d\mathcal H^{N-1}\,.$$
     Adding both equalities we get, $$\int_{\Omega} (T_k(\overline u)-T_k( u))^+(f(\overline u)-f(u))=-\int_{\Omega}(\overline z- z,D(T_k(\overline u)-T_k(u))^+)$$$$-\int_{\partial \Omega}(1-[z,\nu])(T_k(\overline u)-T_k( u))^+\,d\mathcal H^{N-1}.$$
  Finally, since in view of estimate \rife{acota},
  $$
  (\overline \z-\z, D(T_k(\overline u)-T_k(u)   )^+)(B)\geq 0,
  $$
 for any borel set $B\subseteq \Omega$ and using that $\|\z\|_\infty,\|\overline\z\|_\infty\leq 1$, we get $$\int_{\Omega} (T_k(\overline u)- T_k(u))^+(f(\overline u)-f(u))\,dx\leq 0.$$\bk Therefore, letting $k\to\infty$ and since $f$ is increasing, we obtain the desired result.
  \end{proof}
Note that, in particular, if we apply Theorem \ref{maximal} to two large solutions we get the following result:
\begin{corollary}
    \label{unique} If $f$ is an increasing function, there exists at most one large solution of \eqref{maineq}.
  \end{corollary}

\section{Existence result with {a} general absorption term} \label{sec.p1}

In this section we address to  the analysis of \eqref{maineq} under different regularity conditions on the domain $\Omega$.

In order to do that, we first address to the case of $f(s)=s$; i.e: problem \rife{maineqv}.
 The existence of solutions to \eqref{maineqv} (and uniqueness in the case of $\Omega$ having a smooth boundary, see Section \ref{ibc}) will follow in a quite standard way from some recent tools available in the literature. On the other hand, if $\Omega$ is a convex body, then the solutions can be explicitly constructed even if the domain does not satisfy any further regularity condition (i.e. a uniform interior ball condition).

 \medskip

 This permits to obtain existence of solutions to \eqref{maineq} under very general conditions on the absorption term $f$. Our analysis shows that the sufficient condition on $f$ in \eqref{maineq} for obtaining a solution is that $f$ is  increasing in $\R$ \bk  and $f^{-1}$ is increasing and Lipschitz in the domain of the solution to \eqref{maineqv}.

\medskip

{We begin with the following definition and example of existence of large solutions for a very specific class of domains:

\begin{definition}
  We say that a bounded convex set $E$ of  class $C^{1,1}$ \bk is calibrable if there exists a vector field $\xi\in L^\infty(\R^N,\R^N)$ such that $\|\xi\|_\infty\leq 1$, $(\xi,D\chi_E)=|D\chi_E|$ as measures,  and $$-\div\xi=\lambda_E\chi_E \quad {\rm in \ } \mathcal D'(\R^N) $$ for some constant $\lambda_E$. In this case \cite[page 6]{acc}, $\lambda_E=\frac{Per(E)}{|E|}$ and integrating by parts in $E$ it is easily seen that $[\xi,\nu^E]=-1$, $\mathcal H^{N-1}-a.e$ in $\partial E$.
\end{definition}

{As it is proved in \cite[Theorem 9]{acc}, a bounded and convex set $E$ is calibrable if and only if the following condition holds: $$(N-1)\|{\bf H}_E\|_\infty\leq \lambda_E=\frac{ Per(E)}{|E|},$$ where ${\bf H}_E$ denotes the ($\mathcal H^{N-1}$-a.e. defined) mean curvature of $\partial E$. In particular, if $E=B_R(0)$, for some $R>0$, then $E$ is calibrable.}

\begin{example}\label{ex1}
  If $\Omega$ is a calibrable set, then $v=\frac{Per(\Omega)}{|\Omega|}$ is the large solution to \eqref{maineqv}. It suffices to take the restriction to $\Omega$ of the vector field in the definition of calibrability; i.e.: $\z:=-\xi\res_\Omega$, since
   $$\begin{array}{l}\dys  (\z,Dv)(\Omega)\stackrel{\eqref{Green}}=-\int_\Omega \left(\frac{Per(\Omega)}{|\Omega|}\right)^2\,dx\\\\\dys +\int_{\partial\Omega}-[\xi,\nu^\Omega]\frac{Per(\Omega)}{|\Omega|}\,d\mathcal H^{N-1}=0= |Dv|(\Omega).\end{array}$$
\end{example}}

We {next follow} with the proof  of the existence of a large solutions  to \rife{maineqv}  when the domain $\Omega$ {is smooth enough}.

\subsection{The case of the domain verifying a uniform interior ball condition}\label{ibc}

Let $\Omega$ satisfying a uniform interior ball condition: i.e. there exists $s_\Omega>0$ such that for every $x\in\Omega$ with $dist(x,\partial\Omega)<s_\Omega$, there is $z_x\in\partial\Omega$ such that $|x-z_x|=dist(x,\partial \Omega)$ and $B(x_0,s_\Omega)\subset\Omega$ with $x_0:=z_x+s_\Omega\frac{x-z_x}{|x-z_x|}$.  In the same way, one can define the  uniform exterior ball condition by replacing $\Omega$ with $\R^N\setminus\Omega$.  {As is proved in  \cite[Corollary 3.14]{AlMaz} a domain with compact boundary is of class $C^{1,1}$ if and only if it satisfies both a uniform interior ball condition and an exterior one. This result {is implicitly}  used in Section \ref{Secpto1}}.
From now on, $s_\Omega$ will denote the radius of the uniform interior ball condition corresponding to $\Omega$.
\medskip
In \cite{MP}  an operator $\mathcal A$ associated to the elliptic problem \begin{equation}
   \left\{\begin{array}
    {cc} \displaystyle -\div\left(\frac{Dv}{|Dv|}\right)=w\quad & {\rm in \ } \Omega \\ \\ v=+\infty & {\rm on \ }\partial\Omega\,,
  \end{array}\right.
\end{equation}
is defined. More precisely, we have the following: \bk
\begin{definition}\label{52}

    We say that $(v,w)\in \mathcal A$ iff $v,w\in L^1(\Omega)$, $v\in TBV(\Omega)$ and there exists $\z\in X(\Omega)$ with $\|\z\|_{\infty}\leq 1$, $w=-\div \z $ in $\mathcal D'(\Omega)$ such that $$[\z,\nu]=1\,,\qquad \mathcal H^{N-1}-{\rm a.e \ in \ } \partial\Omega$$ and
    \begin{equation}\label{eqpairing} (\z,D S(v))(\Omega)=|DS(v)|(\Omega)\,,\quad {\rm for \ all \ } S\in \mathcal P.\end{equation}
  \end{definition}

\noindent The following result holds true (\cite[Theorem 5.2.]{MP}):

\begin{theorem}\label{thA}
    The operator $\mathcal A$ is m-completely accretive in $L^1(\Omega)$ with dense domain.
  \end{theorem}

\begin{theorem}\label{existenceball}
  Let $\Omega$ satisfy a uniform interior ball condition. Then, there exists a unique large solution $v\in BV(\Omega)\cap L^\infty(\Omega)$ to \eqref{maineqv}.
\end{theorem}
\begin{proof}
By the definition of $m$-accretivity (we do not enter into the details, see for instance \cite{crandall},\cite{becr}) and as a consequence of Theorem  \ref{thA}, we get that for any $w\in L^1(\Omega)$, there exists a unique solution of \bk \begin{equation}
  \label{maineqv2} \left\{\begin{array}
    {cc} \displaystyle-\div\left(\frac{Dv}{|Dv|}\right)= w-v\quad & {\rm in \ } \Omega \\ \\ v=+\infty & {\rm on \ }\partial\Omega\,,
  \end{array}\right.
\end{equation}in the sense that $(v,w-v)\in\mathcal A$. We just take now $w=0$ and we get the result by observing that, by \cite[Remark 5.3]{MP}, $v\in BV(\Omega)\cap L^\infty(\Omega)$.
\end{proof}

\subsection{The case of a convex domain}

Let us consider $\Omega$ being a nontrivial convex body in $\R^N$; i.e. a compact convex subset of $\R^N$ with a nonempty interior.

We  recall the approach and several results given in \cite{ac} which we gather together in the next theorem:

\begin{theorem}[\cite{ac}, Proposition 2.4 and Remark 2.3]\label{acc}Consider the problem $$(P)_\lambda:=\min_{F\subseteq \Omega} Per(F)-\lambda|F|$$
  Then, there is a unique convex set $K\subseteq \Omega$ of class $C^{1,1}$ (the Cheeger set, {which is moreover calibrable}, see \cite{ac,ccn} for details) which is a solution of $(P)_{\lambda_K}$ with $\lambda_D:=\frac{Per(D)}{|D|}$  for any $D\subseteq \Omega$. For any $\lambda>\lambda_K$ there is a unique minimizer $\Omega_\lambda$ of $(P)_\lambda$, which is moreover convex, and the function $\lambda\to \Omega_\lambda$ is increasing and continuous and $\Omega_\lambda\to \Omega$ as $\lambda\to\infty$.
\end{theorem}

Let $K$ be the Cheeger set contained in $\Omega$ defined in the previous result. %$\Omega$ is called {\it calibrable} if $\Omega\equiv K$.
For each $\lambda\in (0,+\infty)$ let $\Omega_\lambda$ be the minimizer of problem $(P)_\lambda$. We take $\Omega_\lambda=\emptyset$ for any $\lambda<\lambda_K$. Using the monotonicity of $\Omega_\lambda$ and the fact that $|\Omega\setminus\cup\{\Omega_\lambda : \lambda>0\}|=0$ we may define  the \emph{variational mean curvature} \bk as \begin{equation}\label{varmecu}H_\Omega(x):=\left\{\begin{array}{cc}-\inf\{\lambda : x\in \Omega_\lambda\} & {\rm if \ }x\in\Omega \\ 0 & {\rm if \ } x\in\R^N\setminus\Omega.\end{array}\right.\end{equation}

In \cite{barozzi} (see also \cite[Theorem 2.3]{goma}) it is established that,  if $\Omega$ is a set of finite perimeter, then $\|H_\Omega\|_{1}=Per(\Omega)$ and  $$\displaystyle \int_{\Omega_\lambda} H_\Omega(x)\,dx=-Per(\Omega_\lambda)\,.$$%Moreover, (\cite[Remark 2.5]{goma}), if $\Omega$ verifies a uniform interior ball condition with radius $s_\Omega$, then $\|H_\Omega\|_\infty\leq \frac{N}{s_\Omega}$.

Thanks to this result, in \cite[Lemma 7]{acc} the authors are able to construct  a vector field $\xi_\Omega\in X(\R^N)$ with $\|\xi_\Omega\|_{\infty  }\leq 1$ such that $\div\xi_\Omega=-H_\Omega$ in $L^1(\R^N)$, and$$ (\xi_\Omega,D\chi_{\Omega_\lambda})(\R^N)=Per(\Omega_\lambda)\,,\quad {\rm for \ any \ }\lambda>0.$$

We next follow the construction of the solution to the Cauchy problem for the Total Variation Flow build up in \cite[Theorem 17]{acc} and \cite[Remark 2.5]{ac} in order to obtain a maximal solution to \eqref{maineqv}.

\begin{theorem}\label{excur}
  Let $\Omega$ be a non-trivial convex body and let $H_\Omega$ be the variational mean curvature given by \eqref{varmecu}. Then, $v(x):=-H_\Omega(x)$ is a large solution to \eqref{maineqv}. Moreover, if $\|H_\Omega\|_{\infty}<+\infty$, then, $\Omega$ is of class $C^{1,1}$.
\end{theorem}
\begin{proof}

First of all, we have that $v\in TBV(\Omega)$ since, for $\lambda\leq k$,$$Per([T_k(v)\geq \lambda],\Omega)=Per ([v\geq \lambda],\Omega)=Per(\Omega_\lambda),$$ which is finite. Then, by the coarea formula (\cite[Theorem 3.40]{Ambrosio}), $$|DT_k(v)|(\Omega)=\int_0^k Per([T_k(v)\geq \lambda],\Omega)\ d\lambda<+\infty.$$

  Let $\xi_\Omega$ be the vector field obtained above and {let us }take $\z:=-\xi_\Omega$. Obviously $[v\leq \lambda]=\Omega_\lambda\subseteq\Omega$ and  $\z\in X(\Omega)$. Moreover, the same computations as in \cite[Theorem 17]{acc} show that \begin{equation}
    \label{xinu}[\z,\nu^\Omega]=1\,,\quad \mathcal H^{N-1}{\rm-a.e \ on \ }\partial\Omega \,,{\rm and}
  \end{equation}
  \begin{equation}
    \label{xinuaux} [\z,\nu^{\Omega_\lambda}]=1\,,\quad \mathcal H^{N-1}{\rm-a.e \ on \ }\partial\Omega_\lambda.
  \end{equation}
Finally, by Lemma \ref{anz}[a],
  $$\begin{array}{l}\dys  (\z, DT_k(v))(\Omega)=\int_0^\infty (\z,D\chi_{[T_k(v)\geq \lambda]})(\Omega)\,d\lambda\\\\ \dys=-\int_0^k\int_{(\partial^*[v\geq \lambda])\cap\Omega}[\z,\nu^{[v\geq\lambda]}]d\mathcal H^{N-1}\,d\lambda\\\\ = \dys \int_0^k \int_{(\partial^*[v\geq \lambda])\cap\Omega}[\z,\nu^{[v\leq\lambda]}]d\mathcal H^{N-1}\,d\lambda\\\\ \dys =\int_0^k \int_{(\partial^*[v\geq \lambda])\cap\Omega}[\z,\nu^{\Omega_\lambda}]d\mathcal H^{N-1}\,d\lambda\\\\\stackrel{\eqref{xinuaux}}= \dys \int_0^k Per([v\geq\lambda],\Omega)d\lambda= |D T_k(v)|(\Omega)\,.\end{array}$$

  Together with \eqref{acota}, this proves that $v$ is a large solution to \eqref{maineqv}.
 By Corollary \ref{unique}, it is its unique solution. Finally, suppose by contradiction that $\|H_\Omega\|_\infty<C$. Then, $\Omega_\lambda=\Omega$ for all $\lambda\geq C$. Since $\Omega_\lambda$ is a solution to $(P_\lambda)$, we proceed as in \cite[Proposition 2.7]{acc} to show that the mean curvature of $\Omega$ is bounded which, together with its convexity, proves that $\Omega$ is $C^{1,1}$.
\end{proof}

\begin{remark}\label{infinito}
Let us  emphasize  the following  qualitative property of large solutions.
  Observe that, as a corollary of Theorem \ref{excur} in case $\Omega$ is a nontrivial convex body which is not $C^{1,1}$, then the large solution to \eqref{maineqv} is not bounded. Heuristically, this means that the large solution takes the values $+\infty$ at those points in the boundary of the convex body which do not have a finite mean curvature (e.g. \bk
   at the \lq\lq corners").
\end{remark}

\subsection{Existence of solutions: the general case}
We next show that the existence of solutions to \eqref{maineqv} permits us to show existence of a large solution of \eqref{maineq}  in the case of $\Omega$ being either a {domain satisfying a uniform interior ball condition} or a convex body.
\begin{theorem}\label{exgenf}
  Let $f$ be an everywhere defined increasing function such that $f^{-1}$ is  locally  Lipschitz continuous in $]0,+\infty[$ . Then, there exists a large solution of \eqref{maineq}.
\end{theorem}
\begin{proof}
  Let $v\in TBV(\Omega)$ be a large solution to \eqref{maineqv} obtained in  Theorems \ref{existenceball} and \ref{excur} under different conditions on $\Omega$. Then, there exists $\z\in X(\Omega)$ such that
$${\rm div } \z=v\,,\quad {\rm in \ }\mathcal{D}'(\Omega)\,,$$ $$(\z,DT_k(v))=|DT_k(v)|\,,\quad {\rm  as \ measures \ for \ any \ }k>0 {\rm \
      and \ }$$ $$[\z,\nu]=1\,,\quad  {\rm for \ a.e. \ }x\in \partial\Omega\,.$$

    Now, we take a ball $B_R$, sufficiently large such that $\Omega\subset B_R$ and we let $v_R(x):=\frac{N}{R}$ to be the unique large solution to \eqref{maineqv} in $B_R$ as seen in Example \ref{ex1}. Then, by Remark \ref{reminside}, $v_R$ is a distributional solution to \eqref{maineqv} in $\Omega$. By Theorem \ref{maximal}, then $v(x)\geq v_R(x)=\frac{N}{R}$, a.e. $x\in \Omega$.

\medskip
   So that, if we take $u:=f^{-1}(v)$, by chain's rule in $BV$ (\cite[Theorem 4.4]{Ambrosio}), $u\in TBV(\Omega)$. First of all,$${\rm div } \z=f(u) \,,\quad {\rm in \ } \mathcal{D}'(\Omega).$$

   Secondly, since $T_k(u)=f^{-1}(T_{f(k)}(v))$, by Lemma \ref{anz}(b)\, $$(\z,DT_k(u))=\theta(\z,DT_k(u))|DT_k(u)|=\theta(\z,DT_{f(k)}(v))|DT_k(u)|\stackrel{\eqref{eqmeasures}}=|DT_k(u)|\,,$$ as  measures, which, coupled with $[\z,\nu]=1$ shows that $u$ is a large solution of \eqref{maineq}.
\end{proof}

\begin{remark}
  Observe that the previous theorem shows that there exist solutions to \eqref{maineq} with nonlinearities which do not verify condition (\ref{KOp}) for any $p\geq 1$. One can take for instance $f(s)=log(1+s)$.
\end{remark}

\begin{remark}\label{maxi}
  Let $C\in\R$ be a constant and $\Omega
  $ be $C^{1,1}$. We consider now the case of $f\equiv C$. A trivial integration by parts shows that for $\div \z=v$ and $[\z,\nu]=1$ to hold, then, necessarily $C=\frac{Per(\Omega)}{|\Omega|}$. Then, $\Omega$ must be a calibrable set and $f=\frac{Per(\Omega)}{|\Omega|}$. In this case, large solutions exist. In fact, $v=\tilde C$ for any constant $\tilde C$ is a large solution to \eqref{maineq} in the sense of Definition \ref{defisol}. However, since $f$ is not increasing, Theorem \ref{maximal} does not hold. In this sense, we can say that there is not a maximal solution to \eqref{maineq}, consistent with the case of the $p-$Laplacian for $p>1$. With the same analysis, it can be proved that if $f$ is constant in a nontrivial  interval around $\frac {Per(\Omega)}{|\Omega|}$, then, large solutions are not unique. Therefore, the {strict} monotonicity is also a necessary condition for obtaining uniqueness of large solutions.
\end{remark}

{\begin{remark}\label{existencedirichlet} { Let $g:\partial\Omega\to \R$.}
% and we consider the general Dirichlet problem \begin{equation}\label{eqdirichlet}\left\{\begin{array}
%  {cc} \Delta_{1}u =f(u) & {\rm in \ }\Omega \\\\ u=g & {\rm on \ }\partial\Omega\,.
%\end{array}\right.\end{equation}}
 With the additional hypothesis (with respect to those in Theorem \ref{exgenf}) on $f$  that $f\circ g\in L^1(\partial\Omega)$ and a similar proof of Theorems \ref{exgenf} and \ref{maximal}, one can easily obtain existence and uniqueness of solutions to problem
\begin{equation}\label{eqdirichlet}\left\{\begin{array}
  {cc} \Delta_{1}u =f(u) & {\rm in \ }\Omega \\\\ u=g & {\rm on \ }\partial\Omega\,,
\end{array}\right.\end{equation}with $g\in L^1(\partial\Omega)$ from the unique solutions (see \cite{ACM}) to problem $$\left\{\begin{array}
  {cc} \Delta_{1} v=v & {\rm in \ }\Omega \\\\ v=f(g) & {\rm on \ }\partial\Omega\,,
\end{array}\right.$${ in the sense that $u\in BV(\Omega)$ and there exists $\z\in X(\Omega)$ such that $$f(u)={\rm div } \z\,,\quad {\rm in \ }\mathcal{D}'(\Omega)\,,$$ $$(\z,Du)=|Du|\,,\quad {\rm  as \ measures \
      and \ }$$ $$[\z,\nu]\in {\rm sign} (T_k(g)-T_k(u))\,,\quad  \mathcal H^{N-1}{\rm  \ a.e. \ }x\in \partial\Omega,$$ for any $k>0$.}
To our knowledge, problem \eqref{eqdirichlet} has not been studied in the literature in this generality.

\end{remark}}

\section{Existence of solutions obtained as limit for $p\to 1^{+}$}
\label{Secpto1}

 Here we want to show a stability type approach to  the existence of a large solution for the $1$-Laplace equation with absorption. The result  has a proper independent interest as it highlights the direct connection with standard large solutions  associated to $p$-Laplace type problems with absorption.   To perform the analysis we have to restrict our assumptions on both the domain and the nonlinearity $f$. Let $\Omega$ be a bounded domain of class $C^{2}$.

Concerning $f$ we assume (H3), that is $f(0)=0$,  $f$ is {continuous and }increasing and  there exists  $\overline{q}>0$ and $c>0$ such that
\begin{equation}\label{fabiana}f(s)\geq c s^{\overline{q}}\end{equation}
 for any $s\in [0,\infty)$.
%This assumption is largely satisfied by the models one has in mind including $f(s)=s^{q}$ for $q>p-1$ as well as limit cases of the form $f(s)=\log (1+ t)^{s} t^{p-1}$ for $s>p$.    \bk

We want to  take the  limit when $p\to 1^{+}$ in problem
\begin{equation}
  \label{p-laplace-ff} \begin{cases}
   \Delta_p u_{p}  = f(u_{p})  & \text{in}\ \Omega,\\
 u_{p}=+\infty &\text{on}\ \partial\Omega,
  \end{cases}
  \end{equation}
 to obtain a large  solution of \begin{equation}
  \label{maineq3} \left\{\begin{array}
    {cc} \displaystyle \div\left(\frac{Du}{|Du|}\right)=f(u) \quad & {\rm in \ } \Omega \\ \\ u=+\infty & {\rm on \ }\partial\Omega\,.
  \end{array}\right.
\end{equation}

As $\overline{q}$ is fixed, without loss of generality we can always think about $p<1+\overline{q}$ in order for the Keller-Osserman condition to be satisfied for \emph{any} $p$ near $1$.

{
The following stability result will be proved along this section and it is the main result in it.

\begin{theorem}
 Let $f$ verify (H3). Then, there is a sequence of solutions to \eqref{p-laplace-ff}, $\{u_p\}_p\subset W^{1,p}_{loc}(\Omega)$ such that $u_p$ converges in $L^1_{loc}(\Omega)$, as $p\to 1^+$, to the unique solution to \eqref{maineq3}.
\end{theorem}

\begin{remark}
  Observe that we work on very general conditions on the nonlinearity $f$ which do not guarantee uniqueness of solutions to \eqref{p-laplace-ff} (for instance we do not assume condition \rife{matero} below). Anyway, thanks to  Corollary \ref{unique} uniqueness is achieved in the limit as  $p$ goes to one.
\end{remark}

We first recall the notion of weak solution to \eqref{p-laplace-ff}.}
 \begin{definition}\label{defdl}
 A function $u_{p}\in W^{1,p}_{loc}(\Omega)\cap L^{\infty}_{loc}(\Omega)$ is a weak solution {to} Problem \rife{p-laplace-ff}  if $T_{k}(u_{p})=k$ on $\partial\Omega$, for any $k>0$, \bk and
\begin{equation}  \label{defdleq}
 \int_{\omega} |\nabla u_{p}|^{p-2}\nabla u_{p} \cdot \nabla v +\int_{\omega}  f(u_{p}) v =0
 \end{equation}
 for any $v\in W^{1,p}_{0}(\omega)$ and $\omega\subset\subset\Omega$.
 \end{definition}
  Existence {and uniqueness} of  nonnegative solutions  $u_{p}$ {to} this problem are {studied} in \cite{dl} (see also \cite{mat}). In particular, existence is obtained (see \cite[Theorem 3.3]{mat}) under the condition on $f$ to be an increasing and continuous function with $f(0)=0$ and satisfying \eqref{KOp}. Under more restrictions on $f$ (see \cite[Corollary 4.5]{mat})  it is shown that the solution  satisfies \begin{equation}\label{matero}\lim_{x\to\partial\Omega}\frac{u_p(x)}{\Psi_p^{-1}(dist(x,\partial\Omega))}=1\,\end{equation}
uniformly, where
$$
\Psi_p(t):=\int_t^{+\infty} \frac{1}{(p' F(s))^\frac{1}{p}}\,ds\,.
$$ {Here $p'=\frac{p}{p-1}$ is the conjugate exponent of $p$.} {In \cite{mat} it is proved that \eqref{matero} yields uniqueness}.
{Note that, if we proceed in a formal way from \eqref{matero} and we let $p\to 1^+$ we obtain that the limit solution when $p\to 1^+$ must be a bounded function.  However, this argument is purely formal. Moreover, we want to show that this is in fact the case for any nonlinearity $f$ verifying (H3) (which does not imply, in general, the condition on $f$ needed for \eqref{matero} to be true {(see condition (A3) in \cite{mat}).}

\medskip
In order to {perform} our stability argument, we will need some careful local a priori estimates on the solutions $u_{p}$ {to} \eqref{p-laplace-ff}. Since we need these estimates to be nondegenerate as $p$ approaches $1$, we have to explicit all the constants appearing in the calculations in order to control them. For the sake of simplicity, throughout this section $C$ will indicate any positive constant (that may change his value from line to line) that could depend on $N$, $\Omega$, but not on $p$; if needed we will also use symbols as, for instance,  $C_{N,|\Omega|}$ \bk in order to stress the dependence of the constant on $N$ and $|\Omega|$.

As we already mentioned, without loss of generality, we can suppose $p$ small enough. {Note that,} if $p<\frac32$, then we can apply a \emph{controlled Sobolev inequality} that reads as follows: let $v\in W^{1,p}_{0}(\Omega)$, then
\begin{equation}\label{sobo}
\into |\nabla v|^{p}\geq \left(\frac{2N-3}{3N-2}\right)^{2} \left(\into|v|^{p^{\ast}}\right)^{\frac{p}{p^{\ast}}},
\end{equation}
 for any $p\in(1,\frac32)$ {with $p^*$ being the Sobolev conjugate of $p$} (see for instance   \cite{ads}).\bk

\subsection{Basic a priori estimates} In order to obtain local estimates we need to construct suitable cut-off functions. To do this we will use a  technical lemma whose proof can be found in \cite[Lemma $1.1$]{leoni}.

%First of all we will exploit the fact that the generalized Keller-Osserman condition \rife {KOp} turns out to imply the following assumption on the nonlinearity $f$ at infinity
%$$
%\int^{\infty} ds/(s f(s))^{\frac1p}<\infty,
%$$
%since
%$$
%F(t)^{\frac1p}=\left(\int_{0}^{t}f(s)ds\right)^{\frac1p}= \left(\int_{0}^{1}tf(st)ds\right)^{\frac1p}\leq (tf(t))^{\frac1p}.
%$$

\begin{lemma}\label{tec}
Let $f$ be an increasing function such that $f(0)=0$, and satisfying \rife{fabiana}, and let $K$ be a positive constant. Then there exists a smooth function $\varphi: [0,1]\mapsto[0,1]$, with $\varphi(0)=\varphi'(0)=0$, $\varphi(1)=1$, such that
$$
t^{p}\frac{\varphi'(\sigma)^{p}}{\varphi(\sigma)^{p-1}} \leq \frac{1}{K} t f(t)\varphi(\sigma) +1,
$$
for any $\sigma\in [0,1]$, $t\geq 0$.
\end{lemma}
\begin{remark}
Observe that, a priori, $\varphi$ can depend on $p$. In fact, Lemma \ref{tec} is nothing but a generalization of Young's Inequality. In the model case in which $f(s)=c s^{\overline{q}}$, with $\overline{q}>p-1$, then $\varphi(\sigma)= \sigma^{\frac{p(\overline{q}+1)}{\overline{q}-p+1}}$.
\end{remark}

\medskip

For $0<r<R\leq 1$, we will consider cut-off functions $\xi$ on balls $B_{r}\subset B_{R}\subset\subset\Omega$, that is, smooth functions in $C^{1}_{0}(B_{R})$, such  that $0\leq \xi\leq 1 $ and $\xi\equiv 1$ on $B_{r}$.  Observe that, if we set $\eta:=\varphi (\xi)$, where $\varphi$ is given by Lemma \ref{tec}, then $\eta$ is a cut-off function for $B_{r}$ as well in $B_{R}$.

\medskip

\noindent{\it Local energy estimate.}
 Here we prove a local  estimate for $\nabla u_{p}$ in $(L^{p}_{loc}(\Omega))^{N}$.

\begin{theorem}\label{5.3}
Let $u_{p}$ be a weak solutions {to} \rife{p-laplace-ff}. Then
\begin{equation}\label{energy}
\int_{B_{r}}|\nabla u_{p}|^{p}\leq C_{r,R}\,.
\end{equation}

\end{theorem}

\begin{proof}
Let $\xi$ be a cut-off function for $B_{r}$  in $B_{R}$ such that $$
|\nabla \xi|\leq \frac{C}{R-r},
$$ $\varphi$ be as in Lemma \ref{tec}, and consider the cut-off function $\eta:=\varphi(\xi)$. We can take $v=\up \eta$ as test in {\rife{defdleq}} and we obtain
$$
\intr |\nabla \up|^{p}\eta +\intr f(\up)\up\eta = \intr |\nabla \up|^{p-2}\nabla \up\nabla\xi \varphi'(\xi) \up.
$$
Now, also using Young's inequality, we have
$$
\begin{array}{l}
\dys\left|\intr |\nabla \up|^{p-2}\nabla \up\nabla\xi \varphi'(\xi) \up\right| \\\\ \dys \leq  \intr |\nabla \up|^{p-1}|\nabla\xi| \varphi'(\xi) \up
\dys \leq \frac{C}{R-r}\intr  |\nabla \up|^{p-1}\varphi(\xi)^{\frac{p-1}{p}} \frac{\varphi'(\xi)}{\varphi(\xi)^{\frac{p-1}{p}} } \up \\\\ \dys \leq  \frac{C\vare}{p'(R-r)}\intr  |\nabla \up|^{p}\eta  + \frac{C}{\vare^{p-1}p(R-r)}\intr \frac{\varphi'(\xi)^{p}}{\varphi(\xi)^{{p-1}} } \up^{p} \\\\ \dys  \leq  \frac{1}{2}\intr  |\nabla \up|^{p}\eta  \dys + \frac{C}{(R-r)}\intr \frac{\varphi'(\xi)^{p}}{\varphi(\xi)^{{p-1}} } \up^{p}\,,
\end{array}
$$
where, in the last inequality, we also choose $\vare=\frac{(R-r)}{2 C} $.

Therefore, we have
\begin{equation}\label{stim1}
\frac12 \intr |\nabla \up|^{p}\eta +\intr f(\up)\up\eta \leq  \frac{C}{(R-r)}\intr \frac{\varphi'(\xi)^{p}}{\varphi(\xi)^{{p-1}} } \up^{p}\,.
\end{equation}
Now, we apply Lemma \ref{tec} with $t=u_{p}$ and $K= \frac{R-r}{C}$, in order to obtain
$$
\begin{array}{l}
\dys \frac{C}{(R-r)}\intr \frac{\varphi'(\xi)^{p}}{\varphi(\xi)^{{p-1}} } \up^{p}\leq \intr f(u_{p})u_{p}\eta +\frac{C}{R-r}|B_{R}|,
\end{array}
$$
which, together with \rife{stim1}, yields the desired result since $\eta\equiv 1$ on $B_{r}$.

\end{proof}

We will also need the following \emph{global} $BV$ bound on the truncations of $u_{p}$.

\begin{lemma}\label{globT}
  Let $u_{p}$ be  a  weak solution of \rife{p-laplace-ff}. Then
\begin{equation}\label{T-energy}
\int_{\Omega}|\nabla T_k(u_{p})|\,dx \leq  C(kf(k))^\frac{1}{p}\,.
\end{equation}
\end{lemma}
\begin{proof}
  Fix $k\in [0,+\infty[$ and let $\Omega_{p,k}:=\{x\in\Omega : u_p(x)\leq k\}$.  First we prove that $\Omega_{p,k}\subset\subset\Omega$. In fact, suppose by contradiction that this is not the case. So that, the exist $x\in \partial\Omega\cap\overline{\Omega_{p,k}}$ and a sequence $\{x_{n}\}\subset\Omega_{p,k}$, with $x_{n}$ that converges to $x$. Now, since
  $u_{p}(x_{n})\leq k$ we deduce that
  $$
  \Psi_{p}(u_{p}(x_{n}))\geq \Psi_{p}(k),
  $$
  that is
  $$
    \frac{\Psi_{p}(u_{p}(x_{n}))}{{ dist}(x_{n},\partial \Omega)}\geq \frac{\Psi_{p}(k)}{{dist}(x_{n},\partial \Omega)},
  $$
   where the right hand side of the previous inequality diverges as $x_{n}$ approaches $x$. This is a contradiction since
  $$
\lim_{x_{n}\to x}  \frac{\Psi_{p}(u_{p}(x_{n}))}{{ dist}(x_{n},\partial \Omega)}=1,
  $$
  as proved in \cite[Theorem 4.4]{mat}.

  \medskip

  Now we are allowed to take $T_k(u_p)-k$ as test function in \eqref{p-laplace-ff}. Therefore,
  $$
  \begin{array}{l}\dys \int_\Omega |\nabla T_k(u_p)|^p\,dx=\int_{\Omega_{p,k}}|\nabla T_k(u_p)|^p\,dx \\ \\ \dys=\int_{\Omega_{p,k}}f(u_p)(k-T_k(u_p))\,dx\leq f(k)k|\Omega|. \end{array}$$ Finally, by H\"older's inequality,  $$\int_{\Omega}|\nabla T_k(u_{p})|\,dx\leq\left(\int_\Omega |\nabla T_k(u_p)|^p\,dx\right)^\frac{1}{p}|\Omega|^\frac{1}{p'}\leq C (kf(k))^\frac{1}{p}\bk\,.$$
\end{proof}

\medskip

\noindent{\it Local boundedness of the solutions.}
 First, we prove an upper bound for solutions to \eqref{p-laplace-ff} in case the domain is a ball:

\begin{lemma}\label{ballbound}
 {Let $\Omega=B_R$. Then, there is  a weak solution to \eqref{p-laplace-ff}, $u_p\in W^{1,p}_{loc}(B_R)\cap L^\infty_{loc}(B_R)$  such that  \begin{equation}\label{ballboundeq}u_p\leq \Psi_p^{-1}\left(\frac{1}{p'N}\left(R-\left(\frac{|x|^p}{R}\right)^\frac{1}{p-1}\right)\right).\end{equation}}
\end{lemma}
\begin{proof}{The existence and the stated regularity of the weak solution follows from \cite[Theorem 3.3]{mat}. However, we need to recall how this weak solution is obtained (see \cite{mat} for the details). We let $u_{p,n}\in W^{1,p}(\Omega)$ to be the unique solution to \begin{equation}
  \label{dirichaux} \left\{\begin{array}{cc}\Delta_p u_{p,n}=f(u_{p,n}) & {\rm in \ } \Omega \\ u_{p,n}=n & {\rm on \ }\partial\Omega\end{array}\right.
\end{equation}Then, since the sequence $\{u_{p,n}\}$ is increasing with respect to $n$, it converges  pointwise \bk to a function $u_p$. This function $u_p$ is a weak solution to \eqref{p-laplace-ff} with the stated regularity.}

 For \eqref{ballboundeq}, we closely follow the proof of \cite[Lemma 4.1]{lieber}. We let $$\tilde F(r):=\frac{f\circ \Psi_p^{-1}(r)}{(p'F\circ \Psi_p^{-1}(r))^\frac{1}{p'}}.$$ Then, $w_p:=\Psi_p(u_p)$ is a {distributional} solution to \begin{equation}\label{lieberinverse}
    \left\{\begin{array}
      {cc} \Delta_p w_{p}=\tilde F(w_{p})[|\nabla w_{p}|^p-1]\quad & {\rm in \ }B_R \\ \\ w_{p}=0 \quad & {\rm on \ }\partial B_R\,
    \end{array}\right..
  \end{equation}

   We consider now $w_p^0$ to be the solution to \begin{equation}\label{lieberaux}
    \left\{\begin{array}
      {cc} \displaystyle\Delta_p w_{p}^{0}=-\frac{N^{2-p}}{R}\quad & {\rm in \ }B_R \\ \\ w_{p}^{0}=0 \quad & {\rm on \ }\partial B_R\,.
    \end{array}\right.
  \end{equation}Then, $w_p^0$ is explicitly given by (see e.g. \cite{talenti}) $$w_p^0=\frac{1}{p'N}\left(R-\left(\frac{|x|^p}{R}\right)^\frac{1}{p-1}\right).$$
  We next show that $w_p^0$ is a subsolution to \eqref{lieberinverse}.

  \medskip
  In the case $f(s)=cs^{q}$,  a direct computation shows that $$\tilde F(r)=\frac{(p-1)(q+1)}{(q+1-p)  \bk  r}.$$ Therefore, $$\Delta_p w_p^0+\tilde F(w_p^0)(1-|\nabla w_0|^p)$$$$=-\frac{N^{2-p}}{R}+\frac{(q+1)p^\frac{1}{p'}N}{(q+1-p)\left(R-\left(\frac{|x|^p}{R}\right)^\frac{1}{p-1}\right)}
  \left(1-\frac{1}{N^p}\left(\frac{|x|}{R}\right)^{p'}\right)$$$$\geq -\frac{N^{2-p}}{R}+\frac{N}{\left(R-\left(\frac{|x|^p}{R}\right)^\frac{1}{p-1}\right)}
  \left(1-\frac{1}{N^p}\left(\frac{|x|}{R}\right)^{p'}\right)$$$$=\frac{N}{R}\left(\frac{\left(1-\frac{1}{N^p}\left(\frac{|x|}{R}\right)^{p'}\right)}
  {\left(1-\left(\frac{|x|}{R}\right)^{p'}\right)}-N^{1-p}\right)\geq 0.$$

From here, we  can proceed exactly as in  \cite[Lemma 4.1]{lieber} to conclude that $w_p\geq w_p^0$ in $B_R$. Then, \begin{equation}\label{bound} u_p=\Psi_p^{-1}(w_p)\leq\Psi_p^{-1}(w_p^0)\,.\end{equation}

\medskip

In the case of a general $f$ satisfying \rife{fabiana} we argue by comparison. We consider the approximating solutions {to} problem   \rife{p-laplace-ff}, that is the weak solutions $u_{p,n}$ to problem \eqref{dirichaux}
that exist and are unique (see for instance \cite[Theorem A]{mat}).
Observe that thanks to \rife{fabiana} we have that $u_{p,n}$ is a subsolution to problem
$$
 \begin{cases}
   -\Delta_p v_{p,n}  + cv_{p,n}^{\overline{q}} = 0 & \text{in}\ \Omega,\\
 v_{p,n}=n &\text{on}\ \partial\Omega,
  \end{cases}
$$
so that, by standard comparison, $u_{p,n}\leq v_{p,n}$, and the bound \rife{bound} is obtained as $n$ goes to $\infty$ as $u_{p}$ is the a.e. limit of $u_{p,n}$.
\end{proof}

\medskip

\noindent{\it Local $BV$ estimate and local estimate on the vector field.} We are in the position to give the essential estimates in order to pass to the limit in the approximating problems \rife{p-laplace-ff}. We prove the following:
\begin{theorem}
 Let $0<r<R$, then it holds,
\begin{equation}\label{lbv}
\int_{B_{r}} |\nabla u_{p}|\leq \tilde{C}_{r,R},
\end{equation}
and for any $1<q<p'$,
\begin{equation}\label{lvf}
\int_{B_{r}} |\nabla u_{p}|^{(p-1)q}\leq \tilde{C}_{r,R,q}\,.
\end{equation}
\bk
\end{theorem}
\begin{proof}
Using H\"older{'s} inequality  and \rife{energy}, we have
$$
\int_{B_{r}} |\nabla u_{p}|\leq \left( \int_{B_{r}}|\nabla u_{p}|^{p}  \right)^{\frac1p}|B_{r}|^{\frac{1}{p'}}\leq \tilde{C}_{r,R}.
$$
Furthermore, in the same way, we have, for {any $q<p'$},
$$
\int_{B_{r}} |\nabla u_{p}|^{(p-1)q}\leq \left(\int_{B_{r}} |\nabla u_{p}|^{p}  \right)^{\frac{q}{p'}}|B_{r}|^{\frac{p'-q}{p'}}\leq \tilde{C}_{r,R,q}.
$$
 Observe that, by virtue of Theorem \ref{5.3}, the constants $\tilde{C}_{r,R}$ and $\tilde{C}_{r,R,q}$, that in principle do depend on $p$, are uniformly controlled, and so they can be chosen to be independent of  this parameter as $p$ approaches $1$.

\end{proof}
\color{black}
\subsection{Passage to the limit}

 Observe that, from both the $L_{\loc}^{\infty}$ bound on $u_{p}$   and \rife{lbv} we deduce that, in particular, the sequence $\{u_{p}\}$ is locally bounded in $W^{1,1}(\Omega)$, so that we can find a subsequence, not relabeled, such that $u_p$ converges in $L^1_{loc}(\Omega)$ and a.e. in $\Omega$ to a function $u\in L^1_{loc}(\Omega)$.
 As a first step, we need to explicit the  $L^{\infty}$ bound on $u_{p}$ in order to obtain a global bound on the limit  $u$ .
\medskip

{\noindent\it Global $L^\infty$ bound on $u$.} Let $s_\Omega$ denote the radius given by the uniform interior ball condition. Then, we can cover $\Omega$ with interior balls with radius bigger than or equal to $s_\Omega$. By Lemma \ref{ballbound}, we have that $$u_p\leq \Psi_p^{-1}\left(\frac{1}{p'N}\left(s_\Omega-\left(\frac{|x-x_0|^p}{s_\Omega}\right)^\frac{1}{p-1}\right)\right)\quad {\rm for \ any \ } x\in B_{s_\Omega}(x_0)\subset\Omega.$$

Let us consider, without loss of generality, the case $f(s)=s^q$. Reasoning as before, the result for a general $f$ satisfying \rife{fabiana} will  easily follow by comparison.

In this case we have,
$$\Psi_p^{-1}(s)=\left(\left(\frac{(q+1)}{p'}\right)^\frac{1}{p}\frac{{ p}}{(q+1-p)s}\right)^\frac{p}{q+1-p}.$$

Then, $$u_p\leq
\left(\frac{(q+1)^\frac{1}{p} { p}{p'}^\frac{1}{p'} N}{(q+1-p)\left(s_\Omega-\left(\frac{|x-x_0|^p}{s_\Omega}\right)^\frac{1}{p-1}\right)}\right)^\frac{p}{q+1-p}\stackrel{p\to 1^+}\to \left(\frac{(q+1)N}{q\ s_\Omega}\right)^\frac{1}{q}\,,$$
that is,
$$u\leq \left(\frac{(q+1)N}{q\ s_\Omega}\right)^{\frac1q}\,,\quad {\rm a.e \ } x\in\Omega.$$
In the general case we have,
 \begin{equation}\label{exactbound}f(u)\leq \frac{(q+1)N}{q\ s_\Omega}\,,\quad {\rm a.e \ } x\in\Omega.\end{equation}

\medskip
{\noindent\it Convergence of the term $|\nabla u_p|^{p-2}\nabla u_p$.}
 Observe that by \eqref{lbv} and \eqref{T-energy} $u\in TBV_{loc}(\Omega)$. As, by what we {have just showed},  $u\in L^\infty(\Omega)$ we deduce that  $u\in BV(\Omega)$.

Furthermore, let $\omega\subset\subset\Omega$. We have that $|\nabla u_p|^{p-2}\nabla u_p$ is weakly  relatively compact in $L^1(\omega;\re^N)$. This is an easy consequence of \rife{lvf}. \bk  In particular,  we may assume that there exists $\z_\omega\in L^1(\omega,\R^N)$ such that $$|\nabla u_p|^{p-2}\nabla u_p\rightharpoonup \z_\omega \,,\quad {\rm as \ } p\to 1^+\, {\rm \ \   weakly \ \bk in \ } L^1(\omega,\R^N).$$ \bk
Following the proof of \cite[Lemma 1]{ACM-N}, we can prove that $\|\z_\omega\|_\infty\leq 1$. Moreover,  by a diagonal argument we can find $\z\in L^{\infty}(\Omega,\R^N)$ with $\|\z\|_\infty\leq 1$ and a subsequence (not relabeled) such that
$$|\nabla u_p|^{p-2}\nabla u_p\rightharpoonup \z \,,\quad {\rm as \ } p\to 1^+\, {\rm \ \ weakly \ in \ } L^1_{loc}(\Omega,\R^N).$$
On the other hand, taking $\varphi\in C_0^\infty(\omega)$ in \eqref{defdleq} and letting $p\to 1^+$ we obtain that $$\dive \z=f(u)\,\quad {\rm in \ } \mathcal D'(\omega).$$Finally, using \eqref{exactbound}, we deduce that $$\dive \z=f(u)\,\quad {\rm in \ } \mathcal{D}'(\Omega).$$

We now use the lower semicontinuity in $L^1(\Omega)$ (see for instance \cite{Anzellotti1}) of the energy functional defined by $$\mathcal F_k(v):=\left\{\begin{array}
  {cc} \displaystyle \int_\Omega |Dv|+\int_{\partial\Omega} |k-v|\,d\mathcal H^{N-1} & {\rm if  \ } v\in BV(\Omega) \\ \\+\infty & {\rm otherwise}
\end{array}\right.$$
 As we already pointed out,  Lemma \ref{globT} allow{s} us to deduce the relative strong compactness in $L^{1}(\Omega)$ of $T_{k}(u_{p})$, so that, reasoning as in Lemma \ref{globT} and using Young's inequality, we have that   $$\int_\Omega |DT_k(u)|+\int_{\partial\Omega} |k-T_k(u)|\,d\mathcal H^{N-1}\leq \liminf_{p\to 1^+}\int_\Omega |DT_k(u_p)|$$
$$
\leq \liminf_{p\to 1^+}\frac{1}{p}\int_\Omega |DT_k(u_p)|^{p} + \lim_{p\to 1^+} \frac{|\Omega|}{p'}
$$
$$=\liminf_{p\to 1^+}\frac{1}{p}\int_\Omega f(u_p)(k-T_k(u_p))\,dx=\int_\Omega f(u)(k-T_k(u))\,dx$$$$=\int_\Omega \dive \z (k-T_k(u))\,dx=\int_\Omega -(\z,DT_k(u))+\int_{\partial\Omega}[\z,\nu](k-T_k(u))\,d\mathcal H^{N-1}$$$$\leq \int_\Omega |DT_k(u)|+\int_{\partial\Omega}|k-T_k(u)|\,d\mathcal H^{N-1}.$$\bk

From here, since the measures $|DT_k(u)|$ are supported in $\Omega$, letting $k\geq f^{-1}(\frac{(q+1)N}{q s_\Omega})$ we obtain that $$ (\z,Du)(\Omega)=|Du|(\Omega)\, {\rm \ and}$$$$[\z,\nu]=1\,,\qquad \mathcal H^{N-1}-{\rm a.e. \ in \ }\partial\Omega .$$Therefore, by \eqref{acota} {and Remark \ref{largemax}}, $u$ is a large solution of \eqref{maineq3} in the sense of Definition 3.2.\bk

\begin{remark}\label{rem-g}
Observe that, a straightforward modification of our arguments in Section \ref{sec.p1} allows us to easily extend the result of existence and uniqueness of large solutions\ to a larger class of  nonhomogeneous problems of the type
$$
\left\{\begin{array}
    {cc} \displaystyle -\div\left(\frac{Du}{|Du|}\right)+f(u)=g(x) \quad & {\rm in \ } \Omega \\ \\ u=+\infty & {\rm on \ }\partial\Omega\,,
  \end{array}\right.
$$
with $g\in L^{1}(\Omega)$ (cfr. with \rife{maineqv2}).  On the other hand, in order to deal with the stability result contained in the last section we have to restrict the class of data  $g$ to $L^{m}(\Omega)$, with $m>N$ in order to get locally bounded approximating solutions.

\end{remark}

We would like to conclude by showing that the upper bounds obtained in \eqref{exactbound} and \eqref{exactboundexp} are or can be taken to be the optimal ones in some particular cases.

\begin{remark}
Let us consider, as a model, the case of $f(s)=u^q$ for some $q>\frac{1}{N-1}$. With a small modification of our argument we can show that,  in this case, we can find an optimal   upper bound for $u$. In fact, if we take $\tilde w_p  ^0$ to be the solution to \begin{equation}\label{lieberaux2}
    \left\{\begin{array}
      {cc} \displaystyle\Delta_p \tilde w_{p}^{0}=-\frac{N^{2-p}}{R}\left(\frac{q+1}{q+1-p}\right)^\frac{1}{p'}\quad & {\rm in \ }B_R \\ \\ \tilde w_{p}^{0}=0 \quad & {\rm on \ }\partial B_R
    \end{array}\right.
  \end{equation} and we follow the proof of Lemma \ref{ballbound}, using $q>\frac{1}{N-1}$, we get that for $p$ sufficiently close to $1$, $$u_p\leq \Psi^{-1}(\tilde w_p^0)\stackrel{p\to 1^+}\to \left(\frac{N}{R}\right)^\frac{1}{q}.$$Therefore,  in this case  the procedure to bound the approximate solutions ends up with the  exact constant  as it can be deduced by Remark \ref{ex1} together with the existence result given by  Theorem \ref{exgenf}\bk. \bk
As a second example, the bound is optimal also for the exponential case $f(s)=e^{s}$ (whose solution exists and it unique as,  in view of Remark \ref{rem-g}, we can subtract and add the constant 1 in order to satisfy our assumptions on $f$). In fact,  in this case,  $\Psi_p^{-1}(r)=\log\left(\frac{(p-1)p^{p-1}}{r^p}\right)$ and $\tilde F(r)=\frac{p-1}{r}$. Then, a direct computation shows that
 \begin{equation}\label{exactboundexp}u_p(r)\leq \log\left(\frac{p^{2p-1}N^p}{(p-1)^{p-1}\left(s_\Omega-\frac{|x|^p}{s_\Omega}^\frac{1}{p-1}\right)^p}\right)\stackrel{p\to 1^+}\longrightarrow \log\left(\frac{N}{s_\Omega}\right)\end{equation}

\end{remark}

\noindent{\bf Acknowledgements.} {First author} acknowledge partial support by project  MTM2012-31103. The authors want to thank J. C. Sabina de Lis for fruitful discussions about the subject and for providing them with reference \cite{mat}.

%%%%%%%%%%%%%%%%%%%%%%%%%%%%%%%%%%%%%%%
%%%%%%%%%%%%%%%%%%%%%%%%%%%%%%%%%%%%%%%
%%%%%%%%%%%%%%%%%%%%%%%%%%%%%%%%%%%%%%%

\end{document}